\documentclass[12pt]{article}

\usepackage[a4paper]{geometry}

\usepackage[utf8]{inputenc}
\usepackage[english]{babel}

\usepackage{biblatex}
\usepackage{amsmath}
\usepackage{amsthm}
\usepackage{amssymb}

\newtheorem{theorem}{Theorem}
\newtheorem{lemma}{Lemma}
\newtheorem{definition}{Definition}
\newtheorem{proposition}{Proposition}
\newtheorem{corollary}{Corollary}

\newtheorem{remark}{Remark}

\title{On Integral Linear Constraints on Convex Cones}
\author{Emil Vladu \\ Alexandre Megretski \\ Anders Rantzer \thanks{The authors are with the Laboratory for Information and Decision Systems at Massachusetts Institute of Technology and the Department of Automatic Control at Lund University. This work was partially supported by the Wallenberg AI, Autonomous Systems and Software Program (WASP) funded by the Knut and Alice Wallenberg Foundation (2024.0459), as well as the European
Research Council (Advanced Grant 834142).}}
\date{}
\begin{document}
\maketitle

\begin{abstract}
    In this paper, we consider integral linear constraints and the dissipation inequality with linear supply rates for certain sets of trajectories confined pointwise in time to a convex cone which belongs to a finite-dimensional normed vector space. Such constraints are then shown to be satisfied if and only if a bounded linear functional exists which satisfies a conic inequality. This is analogous to the typical situation in which a quadratic supply rate over the entire space is related to a linear matrix inequality. A connection is subsequently drawn precisely to linear-quadratic control: by proper choice of cone, the main results can be applied to produce a known $L_1$-gain analogue to the bounded real lemma in positive systems theory, as well as a non-strict version of the Kalman-Yakubovich-Popov Lemma in linear-quadratic control.  
\end{abstract}

\section{Introduction}
In systems theory and control, various constraints on the dynamics of a linear time-invariant (LTI) system can be verified by solving certain algebraic matrix equations or inequalities. Important examples include verifying stability by means of Lyapunov equations, establishing an optimal quadratic cost over all input signals (LQR) \cite{kalman1960contributions} and verifying $L_2$-gain bounds by solving Riccati equations \cite{doyle1988state} or linear matrix inequalities (LMIs) \cite{gahinet1994lmi}. Importantly, symmetry is a recurring feature of these equations and their associated solutions, and the costs featured in the problem formulations are generally quadratic, thus giving rise to the dominating linear-quadratic paradigm in systems theory and control.

By contrast, more recently an area called positive systems theory has gained popularity. Positive systems, which are characterized by nonnegative inputs producing nonnegative outputs, occur naturally in such areas as biology, economy or Markov models, and they possess remarkable properties useful especially for analysis and control of large-scale systems, e.g., \cite{farina2000positive}\cite{rantzer2018tutorial} and the references therein. Examples of such properties include the existence of positive diagonal Lyapunov solutions as both necessary and sufficient for stability, the existence of positive vectors solving entrywise inequalities to verify upper bounds on the $L_1$/$L_\infty$-gain \cite{briat2013robust}\cite{ebihara20111} and positive diagonal solutions \cite{tanaka2011bounded} or nonsymmetric solutions to LMIs \cite{ebihara2014lmi} to certify a given upper bound on the $H_\infty$ norm. Note in particular the recurring absence of symmetry, as well as the occassional occurrence of linear cost functions rather than quadratic ones.

In general, the relationship between the many kinds of equations and inequalities certifying stability and performance for positive systems, and those doing the same for general systems, is complex. However, arguably the most straightforward connection can be seen by considering the following equivalent condition for stability of a Metzler matrix $A \in \mathbb{R}^{n \times n}$ $$\exists p > 0 \; \mathrm{such \; that \; } Ap < 0$$ as well as the one for general matrices $$\exists P \succ 0 \; \mathrm{such \; that \; } A^T P + PA \prec 0,$$ where $>$ and $\succ$ are induced by the nonnegative orthant and positive semidefinite cone, respectively. The common denominator here is an operator $L$ on either $\mathbb{R}^n$ or $\mathcal{S}^n$ such that $\frac{d}{dt} x = L(x)$ leaves the above cones invariant, and under this assumption, asymptotic stability is equivalent to the existence of a positive vector solving a conic inequality. This idea is well-known, e.g., \cite{gowda2009z} for a formalization in a finite-dimensional Hilbert space framework. The assumption is sometimes known as cross-positivity in the literature and has been studied thoroughly over the years, e.g., \cite{schneider1970cross}.

The above way of identifying linear-conic structure in standard Lyapunov stability analysis has very recently been exploited also in other areas of control to provide a unifying framework for various results. For example, \cite{bamieh2024linear} leverages a result on linear-cone duality on general Banach spaces in order to connect differential Riccati equation solutions to various finite-horizon linear-quadratic phenomena such as LQR. Another example is given by \cite{pates2024optimal} which instead passes through the Bellman equation to draw parallels between a recent positive systems result \cite{rantzer2022explicit} and LQR. Matrix dynamical systems associated normally with covariance matrices are central to both works, and have a long history in control, e.g., \cite{gattami2009generalized} where they were used for stochastic control. Note finally the related yet distinct body of literature which also exploits cone theory in the context of control: generalizations of positive systems theory which establish precisely how and for which type of cone its appealing properties manifest in a wider context, e.g., \cite{angeli2003monotone} \cite{papusha2015analysis}\cite{shen2016some}\cite{tanaka2013dc}\cite{shen2017input}.

The main purpose of the present paper is to identify a similar linear-conic structure in some already published results on both positive systems and linear-quadratic theory. More precisely, the satisfaction of various integral linear constraints and the dissipation inequality with linear supply rate for trajectories confined to a cone is shown to be equivalent to the existence of a bounded linear functional satisfying a conic inequality. The latter is perhaps most fruitfully compared to \cite{willems1971least}, in which a similar connection is made between LMIs and linear systems satisfying the dissipation inequality with quadratic supply rate. Further, there is a great literature on integral quadratic constraints, e.g., \cite{megretski1997system} and the references therein, and not unexpectedly integral linear constraints appear also in positive systems theory, e.g., \cite{briat2013robust}. There is a solid theoretical framework for dissipative systems \cite{willems1972dissipative1}, and in particular when the system is linear and the supply rate is quadratic \cite{willems1972dissipative2}. Similarly, there exists a parallel dissipation theory subsequently developed for positive systems \cite{haddad2005stability} in which the supply rate is linear. However, the present paper considers the dissipation inequality only to see what can be said without any nonnegativity requirement on the storage function as in standard dissipativity.

The main results of the present paper are subsequently exploited to derive first a non-strict variant of a known result in positive systems theory on $L_1$-gains \cite{briat2013robust}\cite{ebihara20111} (Proposition \ref{prop:bounded_real_l1}), and second a non-strict version of the Kalman-Yakubovich-Popov (KYP) Lemma \cite{kalman1963lyapunov}\cite{yakubovich1962solution}\cite{popov1961absolute} (Proposition \ref{prop:kyp}). Various versions, generalizations and proofs to the latter have been presented over the years, some of which are algebraic in their nature and others dynamical, e.g., \cite{rantzer1996kalman}\cite{megretski2010kyp}. In this paper, the aim is not swift theorem verification but rather an attempt at shedding additional light on an important result in the control literature, this time from a linear-cone perspective. To this end, a corresponding linear-cone analog to the non-strict KYP Lemma (Theorem \ref{thrm:kyp_linear_cones}) is presented, in which the desired connection between dynamical constraints and conic inequalities is observed in a more basic setting. In connection to this, controllability on a cone, or K-controllability in short, is defined. Note that such notions already exist in the positive systems literature, where their relationship to standard controllability has been studied extensively, e.g., \cite{coxson1987positive}\cite{ohta1984reachability}\cite{valcher1996controllability}\cite{valcher2009reachability} and the references therein. The KYP Lemma now readily follows from Theorem \ref{thrm:kyp_linear_cones} after an application of a crucial rank one decomposition (Theorem \ref{thrm:rank_one}). This latter result is novel to the best of the authors' knowledge and breaks down trajectories on the positive semidefinite cone which satisfy the matrix dynamical system in \cite{bamieh2024linear} to components satisfying standard LTI system dynamics, thereby bridging the gap between the linear-cone and the linear-quadratic domain.

The outline of the paper is as follows: in Section \ref{Section:Preliminaries}, we define and recall relevant mathematical notions. Section \ref{Section:Results} provides the results of the paper and Section \ref{sec:Proofs} the associated proofs. Section \ref{Section:Conclusions} subsequently concludes the paper. Finally, we note that the contents of this paper are based on material in \cite{vladu2025performance}.


\section{Preliminaries} \label{Section:Preliminaries}
In this section, we explain the notation used throughout the paper and recall basic functional analytical concepts.

We denote by $\mathbb{R}$ ($\mathbb{C}$) the set of real (complex) numbers, and by $\mathbb{R}^n$ ($\mathbb{C}^n$) and $\mathbb{R}^{n \times m}$ ($\mathbb{C}^{n \times m}$) the set of $n$-dimensional vectors and $n \times m$-matrices, respectively, with entries in $\mathbb{R}$ ($\mathbb{C}$). The identity matrix will be denoted by $I$, with context determining its dimension. A square matrix $A$ is said to be Hurwitz if all its eigenvalues have negative real part, and Metzler if all its offdiagonal elements are nonnegative. We denote by $\mathcal{S}^n$ the set of symmetric matrices in $\mathbb{R}^{n \times n}$ and by $\mathcal{S}_{+}^n$ the positive semidefinite cone therein which induces the partial order $\succeq$. Similarly, $\mathbb{R}_+^{n}$ is the set of entrywise nonnegative vectors which induces the entrywise partial order $\geq$. For Metzler matrices $A$, it is well known that being Hurwitz is equivalent to $-A^{-1} \geq 0$ as well as the existence of a $p > 0$ such that $Ap < 0$, e.g., \cite{rantzer2018tutorial}.

In this paper, we shall refer to a vector space $X$ over $\mathbb{R}$ equipped with a norm $\lVert \cdot \rVert_X$ as a normed space. Convergence, limits and continuity are defined in the usual $\varepsilon-\delta$ sense as in real analysis. For a linear transformation $L: X \rightarrow Y$ with normed spaces $X, Y$, we say that $L$ is bounded if there exists an $M > 0$ such that $\lVert L(x) \rVert_Y \leq M \lVert x \rVert_X$ for all $x \in X$. We denote by $X^*$ the dual of $X$, i.e., the set of all bounded linear functionals ($x^* : X \rightarrow \mathbb{R}$), and by $B(X, Y)$ the set of all bounded linear transformations $L: X \rightarrow Y$. The adjoint $L^*$ corresponding to $L$ is defined as the transformation $L: Y^* \rightarrow X^*$ such that $L^* (y^*) (x) = y^* (L(x))$ for all $x \in X$. When the norm is induced by an inner product, it is well known by the Riesz representation theorem that $X^*$ can be identified with $X$. For an excellent introduction to the theory of normed spaces, see e.g., \cite{luenberger1997optimization}.

A cone $K \subseteq X$ is a set for which $x \in K$ implies $\alpha x \in K$ for all $\alpha \geq 0$. A convex (and pointed, i.e., $K \cap -K = \{0\}$) cone induces a preorder (partial order) $\succeq_K$ such that $x \succeq_K y$ if and only if $x - y \in K$; $x \succ_K y$ if and only if $x - y \in \mathrm{Int}(K)$, where Int denotes the interior. The associated dual cone is defined as $K^* = \{ x^* \in X^* \mid x^* (x) \geq 0 \; \forall x \in K \}$, and in finite dimensions for closed $K$ the interior of the dual is given by $\mathrm{Int}(K^* ) = \{ x^* \in X^* \mid x^* (x) > 0 \; \forall x \in K , x \ne 0 \}$, noting that $\mathrm{Int}(K^* )$ is nonempty if in addition $K$ is also pointed, e.g., \cite[p. 64]{boyd2004convex}. Examples of closed, convex, pointed cones with nonempty interior are the nonnegative orthant in $\mathbb{R}^n$ and the positive semidefinite cone in $\mathcal{S}^n$. For more on finite-dimensional cones, see e.g., \cite{berman1994nonnegative}\cite{barker1981theory}.

Given a finite-dimensional normed space $X$ and an open interval $I \subseteq \mathbb{R}$, we say that $f: I \rightarrow X$ is differentiable at $t_0 \in I$ if the limit $\lim_{h \rightarrow 0} \frac{f(t_0 + h) - f(t_0 )}{h}$ exists, and we denote it variously by $\dot{f}(t_ 0)$ or $\frac{d}{dt}f(t_0 )$; if the limit exists for all $t_0 \in I$ we say that $f$ is differentiable. In the interest of simplicity, we interpret $\int_I f(t) \; \mathrm{dt}$ in the Riemann sense, for which integration on closed intervals $I$ generalizes in a natural way to $X$-valued functions. Improper integrals are carried out in the principal value sense, i.e., $\int_{- \infty}^{\infty} f(t) \; \mathrm{dt} = \lim_{T \rightarrow \infty} \int_{-T}^{T} f(t) \; \mathrm{dt}$, provided the limit exists. We note that standard intuition applies well, as most of the basic results from real analysis persist in this setting, e.g., \cite{gordon1991riemann}. Two such examples of which we shall make use include the fundamental theorem of calculus and interchanging the order of integration with a linear operator. Finally, given $E, L \in B(X, Y)$, we say that a piecewise continuous $f$ satisfies the differential equation $\frac{d}{dt}E(f) = L(f)$ on some interval $I$ if $E(f(t)) - E(f(t_0 ))= \int_{t_0}^t L(f(\tau )) \; \mathrm{d\tau}$ for all $t, t_0 \in I$ or, equivalently, if in addition to being continuous, $E(f)$ is differentiable and the differential equation holds on those open intervals on which $f$ is continuous. We assume in the remainder of the paper that all functions of time are continuous a.e.

We close this section by recalling the following fact.
\begin{lemma} \label{lem:image_partition}
    Given a matrix $Q \in \mathcal{S}^{n + m}$ such that $Q \succeq 0$ with corresponding partition $$Q = \begin{pmatrix}
        Q_{nn} & Q_{nm} \\ Q_{nm}^T & Q_{mm}
    \end{pmatrix},$$ then $\mathrm{Im} (Q_{nm}) \subseteq \mathrm{Im} (Q_{nn})$.
\end{lemma}
\begin{proof}
    Suppose on the contrary that this were not the case. Then there exists a $z \in \mathbb{R}^m$ such that $Q_{nm} z \notin \mathrm{Im}(Q_{nn}) = \big( \mathrm{Im}(Q_{nn})^\bot \big) ^\bot$. As such, there is $w \in \mathbb{R}^n$ such that $w^T Q_{nn} v = 0$ for all $v \in \mathbb{R}^n$, yet $w^T Q_{nm} z \ne 0$. But this would imply that 
    \begin{equation*}
    \begin{aligned}
    \begin{pmatrix}
        w \\ z
    \end{pmatrix}^T Q \begin{pmatrix}
        w \\ z
    \end{pmatrix} &= w^T Q_{nn} w + 2 w^T Q_{nm} z + z^T Q_{mm} z \\ &= 2 w^T Q_{nm} z + z^T Q_{mm} z < 0
    \end{aligned}
    \end{equation*}
    by choosing the sign of $w$ properly and making it sufficiently large, a contradiction as $Q \succeq 0$.
\end{proof}


\section{Results} \label{Section:Results}
In this section, we present some general results on integral linear constraints on cones. At the heart lies the following finite-dimensional phenomenon for which a non-strict and a strict version are provided.
\begin{theorem} \label{thrm:sep_non-strict}
    Let the finite-dimensional normed spaces $Z, X$ and the convex cone $K \subseteq Z$ be given. Then for any given $m^* \in Z^*$ and $L \in B(Z, X)$ such that $L(K) = X$, the following conditions are equivalent:
    \begin{enumerate}
        \item[(i)] There exists $p^* \in X^*$ such that \begin{equation} \label{eq:lmi}
            L^* (p^* ) - m^* \preceq_{K^*} 0.
        \end{equation}
        \item[(ii)] $m^* (z_0 ) \geq 0$ for every $z_0 \in K$ such that $L(z_0 ) = 0$.
    \end{enumerate}
\end{theorem}
\begin{theorem} \label{thrm:sep_strict}
    Let the finite-dimensional normed spaces $Z, X$ and the closed, convex and pointed cone $K \subseteq Z$ be given. Then for any given $m^* \in Z^*$ and $L \in B(Z, X)$, the following conditions are equivalent:
    \begin{enumerate}
        \item[(i)] There exists $p^* \in X^*$ such that \begin{equation*}
        L^* (p^* ) - m^* \prec_{K^*} 0.
        \end{equation*}
        \item[(ii)] $m^* (z_0 ) > 0$ for every nonzero $z_0 \in K$ such that $L(z_0 ) = 0$.
    \end{enumerate}
\end{theorem}
\begin{proof}
    See Section \ref{sec:Proofs}.
\end{proof}

Although the nature of Theorem \ref{thrm:sep_non-strict} and Theorem \ref{thrm:sep_strict} is finite-dimensional, it does have bearing on functions of time. In order to see this, let $L, X, K$ be given as above, define the three sets 
\begin{equation*}
\begin{aligned}
H_l &= \cup_{[t_0 , t_1]} \scriptstyle \{z: [t_0 , t_1] \rightarrow K \mid z(t) = z_0 \mathrm{\; for \; some \; } z_0 \in K \mathrm{\; such \; that \; } L(z_0) = 0\} \\ 
H_u &= \cup_{[t_0 , t_1]} \scriptstyle \{z: [t_0 , t_1] \rightarrow K \mid \exists \; x: [t_0 , t_1] \rightarrow X \mathrm{\; such \; that \;} \dot{x} = L(z) \} \\
H_v &= \cup_{[t_0 , t_1]} \scriptstyle \{z: [t_0 , t_1] \rightarrow K \mid \exists \; x: [t_0 , t_1] \rightarrow X \mathrm{\; such \; that \;} \dot{x} = L(z) \mathrm{\; and \;} x(t_0 ) = x(t_1 ) \}
\end{aligned}
\end{equation*}
and denote by $H_z$ any subset of all $x$ such that $\dot{x} = L(z)$ given $z \in H_u$. The following then holds.
\begin{corollary} \label{cor:nondym}
    Let the finite-dimensional normed spaces $Z, X$ and the convex cone $K \subseteq Z$ be given. Then for any set $H$ constrained as $H_l \subseteq H \subseteq H_u$ with an associated family $H_z$ and any $m^* \in Z^*$ and $L \in B(Z, X)$ such that $L(K) = X$, the following conditions are equivalent:
    \begin{enumerate}
        \item[(i)] There exists $p^* \in X^*$ such that \begin{equation*}
        L^* (p^* ) - m^* \preceq_{K^*} 0.
        \end{equation*}
        \item[(ii)] There exists a continuous $V: X \rightarrow \mathbb{R}$ with $V(0) = 0$ such that for all $z \in H$ and $x \in H_z$,
        \begin{equation*}
        V(x(t_0)) + \int_{t_0}^{t_1} m^*(z(t)) \; \mathrm{dt} \geq V(x(t_1)).
        \end{equation*}
    \end{enumerate}
    If further $H_l \subseteq H \subseteq H_v$, then the following condition is also equivalent to condition (i) and (ii):
    \begin{enumerate}
        \item[(iii)] For all $z \in H$, $$\int_{t_0}^{t_1} m^* (z(t)) \; \mathrm{dt} \geq 0.$$
    \end{enumerate}
\end{corollary}
\begin{proof}
    See Section \ref{sec:Proofs}.
\end{proof}

Thus far, there has been no mention of dynamics. This changes now as we apply Corollary \ref{cor:nondym} to sets $H$ constrained by differential equations. First, however, we make the following systems theoretical definition in keeping with \cite{willems1971least}.
\begin{definition} \label{def:dissipativity}
        Let the finite-dimensional normed spaces $Z, X$ and the convex cone $K \subseteq Z$ be given. We say that the pair $E, L \in B(Z, X)$ satisfies the dissipation inequality on $K$ w.r.t. $w: K \rightarrow \mathbb{R}$ if there exists a continuous function $V: X \rightarrow \mathbb{R}$ with $V(0) = 0$ such that 
    \begin{equation} \label{eq:dis_ineq}
    V(E(z(t_0))) + \int_{t_0}^{t_1} w(z(t)) \; \mathrm{dt} \geq V(E(z(t_1)))
    \end{equation}
    for all $t_1 \geq t_0$ and all $z : [t_0 , t_1 ] \rightarrow K$ such that $\frac{d}{dt} E(z) = L(z)$ on $[t_0 , t_1 ]$.
\end{definition}

A consequence of Corollary \ref{cor:nondym} is the following.
\begin{corollary} \label{cor:dissipation}
        Let the finite-dimensional normed spaces $Z, X$ and the convex cone $K \subseteq Z$ be given. Then for any $m^* \in Z^*$ and $E, L \in B(Z, X)$ such that $L(K) = X$, the following conditions are equivalent:
    \begin{enumerate}
        \item[(i)] There exists $p^* \in X^*$ such that \begin{equation*}
        L^* (p^* ) - m^* \preceq_{K^*} 0.
        \end{equation*}
        \item[(ii)] $(E, L)$ satisfies the dissipation inequality on $K$ w.r.t. $m^*$.
    \end{enumerate}
    In addition, if condition (ii) holds, then the function $V$ in Definition \ref{def:dissipativity} can be chosen in $X^*$.
\end{corollary}
\begin{proof}
    See Section \ref{sec:Proofs}. 
\end{proof}

Finally, we consider what can be said for sets of trajectories converging to the origin. For this purpose, we define the following cone analog to the standard concept of controllability in keeping with e.g., \cite{valcher1996controllability}.
\begin{definition} \label{def:controllability}
     Let the finite-dimensional normed spaces $Z, X$ and the convex cone $K \subseteq Z$ be given. $E, L \in B(Z, X)$ is then said to be controllable on $K$ if for every $x_0, x_1 \in E(K)$ there is $t_1 \geq 0$ and a continuous $z: [0, t_1 ] \rightarrow K$ such that $\frac{d}{dt} E(z) = L(z)$ and $E(z(0 )) = x_0, E(z(t_1 )) = x_1$.
\end{definition}

The following can be said to constitute a cone analog to the KYP Lemma, cf. Proposition \ref{prop:kyp}.
\begin{theorem} \label{thrm:kyp_linear_cones}
    Let the finite-dimensional normed spaces $Z, X$ and the convex cone $K \subseteq Z$ be given. Suppose now that the pair $E, L \in B(Z, X)$ is controllable on $K$ and that $E(K)$ has nonempty interior. Then for any given $m^* \in Z^*$, the following conditions are equivalent:
    \begin{enumerate}
        \item[(i)] There exists $p^* \in X^*$ such that \begin{equation*}
        L^* (p^* ) - m^* \preceq_{K^*} 0.
        \end{equation*}
        \item[(ii)] For all $z : \mathbb{R} \rightarrow K$ such that $\frac{d}{dt} E(z) = L(z)$ and $\int_{-\infty}^{\infty} \lVert z(t) \rVert_Z \mathrm{dt} < \infty$, $$\int_{-\infty}^{\infty} m^* (z(t)) \; \mathrm{dt} \geq 0.$$
        \item[(iii)] $m^* (z_0 ) \geq 0$ for every $z_0 \in K$ such that $L(z_0 ) = 0$.
    \end{enumerate}
\end{theorem}
\begin{proof}
    See Section \ref{sec:Proofs}.
\end{proof}

We next give some remarks on the above results.
\begin{remark}
    The heart of both Theorem \ref{thrm:sep_non-strict} and Theorem \ref{thrm:sep_strict} is a separating hyperplane argument. Although such statements exist also for infinite-dimensional spaces (Hahn-Banach), one of the sets involved is then required to possess nonempty interior in the case of non-strict inequality, something that does not generally hold in this setting. Similarly, compactness becomes too severe an assumption in the strict case.
\end{remark}
\begin{remark}
    The above results are fundamentally about connecting the existence of elements in the dual that satisfy a conic inequality to integral linear constraints on sets of trajectories confined to a cone. Surprisingly, however, dynamics and differential constraints appear not to play an essential role in this phenomenon: if $(E, L)$ satisfies the dissipation inequality, then a solution to (\ref{eq:lmi}) proves that (\ref{eq:dis_ineq}) holds for many other sets $H$ of trajectories due to Corollary \ref{cor:nondym}, so long as $H_l \subseteq H \subseteq H_u$. Examples include the set of all $z: [t_0 , t_1 ] \rightarrow K$ such that there exists an $\hat{E} \in B(Z, X)$ such that $\frac{d}{dt} \hat{E}(z) = L(z)$, or indeed the dynamically disconnected set $H_u$ itself.
\end{remark}
\begin{remark}
    When $Z = \mathbb{R}^{n + m}$, $X = \mathbb{R}^n$, $L(x, u) = Ax + Bu$ and $E(x, u) = x$, Corollary \ref{cor:dissipation} collapses to a statement about standard LTI systems $\dot{x} = Ax + Bu$ relevant to systems theory and control. The statement can then be compared to \cite{willems1971least}, which connects quadratic supply rates on all of $\mathbb{R}^{n + m}$, as opposed to linear ones on cones, to an LMI, corresponding to the conic inequality (\ref{eq:lmi}). Similarly, Corollary \ref{cor:nondym} is comparable to part of \cite[Theorem 2]{willems1971least} but additionally includes the algebraic condition (i).
\end{remark}
\begin{remark}
    The above results concerning trajectories all involve non-strict inequalities and rely on Theorem \ref{thrm:sep_non-strict}. However, analogous results with strict inequality can be obtained in a similar fashion from Theorem \ref{thrm:sep_strict} instead. The main benefit in this case is that the assumption $L(K) = X$ vanishes, but in exchange additional assumptions on the cone are incurred.
\end{remark}


\subsection{Applications} \label{ref:Applications}
Next, we proceed to apply the above results in order to obtain and thereby connect two seemingly unrelated results in the control literature. In the first result, we choose $K$ as the nonnegative orthant in $\mathbb{R}^{n + m}$ and regain a non-strict variant of a known result for positive systems on $L_1$-gains reminiscent of the bounded real lemma \cite{briat2013robust}\cite{ebihara20111}.
\begin{proposition} \label{prop:bounded_real_l1}
    Let $\gamma > 0$ be given and consider the system $\dot{x} = Ax + Bu$ with $A$ Metzler and $B \in \mathbb{R}_{+}^{n \times m}$ such that $Bu > 0$ for some $u \geq 0$. The following conditions are equivalent:
    \begin{enumerate}
        \item[(i)] $A$ is Hurwitz and $$\sup_{\substack{u \in L_1^m [0, \infty ), u \geq 0 \\ u \ne 0, x(0) = 0}} \frac{\lVert x \rVert _1}{\lVert u \rVert _1} \leq \gamma$$
        \item[(ii)] There exists $p > 0$ such that \begin{equation*} \begin{aligned}
            p^T A + \mathbf{1}_{n}^T &\leq 0 \\
            p^T B - \gamma \mathbf{1}_{m}^T &\leq 0
        \end{aligned}
        \end{equation*}
    \end{enumerate}
\end{proposition}
\begin{proof}
    See Section \ref{sec:Proofs}. 
\end{proof}

In the second result, we choose $K$ as the positive semidefinite cone in $\mathbb{S}^{n + m}$ and recover a non-strict version of the Kalman-Yakubovich-Popov (KYP) Lemma. For this purpose, we shall require the following crucial decomposition of trajectories on $\mathcal{S}_+^{n + m}$.
\begin{theorem} \label{thrm:rank_one}
    Given $A \in \mathbb{R}^{n \times n}$, $B \in \mathbb{R}^{n \times m}$ and $Q: \mathbb{R} \rightarrow \mathcal{S}_{+}^{n + m}$ with piecewise real analytical entries, the following conditions are equivalent:
	\begin{enumerate}
		\item[(i)] $Q$ satisfies
        \begin{equation} \label{eq:lem_matrix_dynamics}
		    \frac{\mathrm{d}}{\mathrm{dt}} \begin{pmatrix} I & 0 \end{pmatrix} Q \begin{pmatrix} I \\ 0 \end{pmatrix} = \begin{pmatrix} A & B \end{pmatrix} Q \begin{pmatrix} I \\ 0 \end{pmatrix} + \begin{pmatrix} I & 0 \end{pmatrix} Q \begin{pmatrix} A^T \\ B^T \end{pmatrix}.
        \end{equation}
        \item[(ii)] There exist $n + m$ functions $x_i: \mathbb{R} \rightarrow \mathbb{R}^n$ and $u_i: \mathbb{R} \rightarrow \mathbb{R}^m$ such that $$Q = \sum_{i = 1}^{n + m} \begin{pmatrix} x_i \\ u_i \end{pmatrix} (x_i^T \; u_i^T )$$ and $\dot{x}_i = Ax_i + Bu_i$ on $\mathbb{R}$ for all $1 \leq i \leq n$ and $x_i = 0$ otherwise.
    \end{enumerate}
\end{theorem}
\begin{proof}
    See Section \ref{sec:Proofs}.
\end{proof}

\begin{remark}
    Note that the regularity assumption imposed on $Q$ in Theorem \ref{thrm:rank_one} is ultimately not necessary, but allows for a more intuitive proof which results in a more explicit construction of the terms in condition (ii). In fact, a similar statement to Theorem \ref{thrm:rank_one} will hold more generally for locally (square-)integrable functions and can be proved using techniques from functional analysis. Conversely, there is no reason for the purpose of this paper that the relevant functions in the previous definitions cannot also be assumed to be well-behaved in this manner.
\end{remark}

As a first application of Theorem \ref{thrm:rank_one}, we connect controllability to K-controllability on $\mathcal{S}_+^{n + m}$. For this purpose, define $E, L \in B(Z, X)$ in the natural way from (\ref{eq:lem_matrix_dynamics}) in Theorem \ref{thrm:rank_one} such that (\ref{eq:lem_matrix_dynamics}) can be written as $\frac{d}{dt}E(Q) = L(Q)$. We then have the following.
\begin{corollary} \label{cor:k_reach}
    $(A, B)$ is controllable if and only if $(E, L)$ is controllable on $\mathcal{S}_{+}^{n+m}$.
\end{corollary}
\begin{proof}
    See Section \ref{sec:Proofs}.
\end{proof}

Theorem \ref{thrm:kyp_linear_cones}, Theorem \ref{thrm:rank_one} and Corollary \ref{cor:k_reach} now together give the following result \cite{willems1971least}\cite{rantzer1996kalman}\cite{megretski2010kyp}.
\begin{proposition} [KYP] \label{prop:kyp}
    Given $A \in \mathbb{R}^{n \times n}$, $B \in \mathbb{R}^{n \times m}$, $M = M^T \in \mathbb{R}^{(n + m) \times (n + m)}$ with $(A, B)$ controllable, the following statements are equivalent:
    \begin{enumerate}
        \item[(i)] There exists a matrix $P \in \mathbb{R}^{n \times n}$ such that $P = P^T$ and $$M + \begin{pmatrix}
            A^T P + PA & PB \\ B^T P & 0
        \end{pmatrix} \preceq 0.$$
        \item[(ii)] For all $x \in \mathbb{C}^n$ and $u \in \mathbb{C}^m$ such that either $x = 0$ or $i\omega x = Ax + Bu$ for some $\omega \in \mathbb{R}$, $$\begin{pmatrix}
            x \\ u
        \end{pmatrix}^T M \begin{pmatrix}
            x \\ u
        \end{pmatrix} \leq 0.$$ 
        \item[(iii)] For all $x \in L_2^n (-\infty , \infty)$ and $u \in L_2^m (-\infty , \infty)$ such that either $x = 0$ or $\dot{x} = Ax + Bu$ on $\mathbb{R}$,
        $$\int_{-\infty}^{\infty} \begin{pmatrix}
            x \\ u
        \end{pmatrix}^T M \begin{pmatrix}
            x \\ u
        \end{pmatrix} \; \mathrm{dt} \leq 0.$$
        \item[(iv)] For all $\omega \in \mathbb{R}$ such that $i \omega$ is not an eigenvalue of $A$, $$\begin{pmatrix}
            (i\omega I - A)^{-1}B \\ I \end{pmatrix}^{*}
        M \begin{pmatrix}
            (i\omega I - A)^{-1}B \\ I \end{pmatrix} \preceq 0.$$
    \end{enumerate}
\end{proposition}
\begin{proof}
    See the end of this section. 
\end{proof}

Before we close the section with a proof to Proposition \ref{prop:kyp}, we provide the following remarks.
\begin{remark}
    Traditionally, the KYP Lemma is regarded as a bridge between state space formalism and the frequency domain, and is usually phrased in terms of an equivalence between the LMI in condition (i) and the set of frequency inequalities in condition (iv). From the perspective of the linear-cone theory advanced in Section \ref{Section:Results}, however, the frequency inequality is perhaps best viewed as part of a transition from the vector quadratic constraint in condition (ii) to the integral quadratic constraint in condition (iii), which incidentally manifests itself through the frequency domain and Parseval's theorem, an $L_2$-specific phenomenon. The cone analog to Proposition \ref{prop:kyp} is given by Theorem \ref{thrm:kyp_linear_cones}, and the above two conditions then correspond to the equilibrium point condition (iii) and the integral linear constraint condition (ii), respectively.
\end{remark}
\begin{remark} \label{rem:kyp_nontrivial}
    The nontrivial part of the KYP Lemma in Proposition \ref{prop:kyp} occurs in the unexpected fact that it is sufficient to consider only a subset of all trajectories $Q(t) \succeq 0$ when verifying the integral linear constraint in Theorem \ref{thrm:kyp_linear_cones}. In particular, this subset consists of those trajectories confined to the rank one part of the boundary of $\mathcal{S}_+^{n + m}$, thereby explaining the quadratic rather than linear form of the integrand in condition $(iii)$. Moreover, this central step corresponding to the direction $(iii) \Rightarrow (i)$ is enabled precisely by Theorem \ref{thrm:rank_one}. Note that this step is frequency independent.
\end{remark}
\begin{remark}
    As a complement to Remark \ref{rem:kyp_nontrivial}, we observe also a second dimension further enforcing it: there appears to exist a correspondence between complex vectors and real vector-valued functions of time. This can be observed already in the parallel conditions (ii) and (iii). In fact, the supplementary results in the well-known proof for the KYP Lemma in \cite{rantzer1996kalman} can be used to obtain a complex vector analog to the crucial Theorem \ref{thrm:rank_one}, in which a real positive semidefinite matrix satisfying an equilibrium point condition, as opposed to a matrix-valued function of time satisfying the corresponding dynamics, is decomposed into a sum of rank one matrices, the components of which are complex vectors satisfying the constraint in condition (ii) in Proposition \ref{prop:kyp}. This subsequently offers a parallel algebraic proof of the KYP Lemma in which the nontrivial consideration of only part of $\mathcal{S}_+^{n + m}$ as in Remark \ref{rem:kyp_nontrivial} occurs instead over the field of complex numbers. This corresponds to the direction $(ii) \Rightarrow (i)$ and is not seen in the proof given by \cite{rantzer1996kalman}, as the same fundamental components are executed in a different order.
\end{remark}
\begin{remark}
    Theorem \ref{thrm:rank_one} says that trajectories confined to $\mathcal{S}_+^{n + m}$ that satisfy system (\ref{eq:lem_matrix_dynamics}) can be decomposed into sums of rank one matrix trajectories, the corresponding vectors of which satisfy $x = 0$ or $\dot{x} = Ax + Bu$. Thus, in some sense, nothing new happens in the interior of the cone for the extended system (\ref{eq:lem_matrix_dynamics}), and dynamics on $\mathcal{S}_+^{n + m}$ can essentially be expressed in terms of the original system. A consequence is also that trajectories satisfying the original dynamics correspond to rank one matrix trajectories on the boundary of $\mathcal{S}_+^{n + m}$.
\end{remark}
\begin{remark}
    Theorem \ref{thrm:rank_one} appears to constitute a bridge between positive systems and cone intuition on the one hand, and linear-quadratic intuition on the other. For instance, it is the phenomenon that transfers linear functionals onto quadratic ones, as in the KYP Lemma. Another example is given by the transfer of standard controllability onto K-controllability as in Corollary \ref{cor:k_reach}.
\end{remark}

\begin{proof} \textbf{Proposition \ref{prop:kyp}}

\underline{$(i) \Rightarrow (ii)$:} Multiply the matrix in the LMI by $(x, u) \in \mathbb{C}^{n + m}$ from the right and $(x, u)^*$ from the left and note that the expression in condition (ii) follows immediately if $x = 0$, and otherwise also since
\begin{equation*}
\begin{aligned}
    \begin{pmatrix}
        x \\ u
    \end{pmatrix}^* \begin{pmatrix}
            A^T P + PA & PB \\ B^T P & 0
        \end{pmatrix}  \begin{pmatrix}
        x \\ u
    \end{pmatrix} &=
     (Ax + Bu)^* P x + x^* P (Ax + Bu) \\ &=
     -i\omega x^* P x + i\omega x^* Px = 0.
\end{aligned}
\end{equation*}

\underline{$(ii) \Rightarrow (iv)$:} Take any such $\omega \in \mathbb{C}$ and multiply the matrix in condition (iv) by $u \in \mathbb{C}^m$ from the right and $u^*$ from the left and note that the quadratic form in condition (ii) is obtained by setting $x :=(i\omega I - A)^{-1} B u$. If $x = 0$, condition (ii) can be invoked directly and if $x \ne 0$, note that the latter implies $i\omega x = Ax + Bu$ so that condition (iv) follows from (ii).

\underline{$(iv) \Rightarrow (iii)$:} Take any $x \in L_2^n (-\infty , \infty )$ and $u \in L_2^m (-\infty , \infty )$ such that $\dot{x} = Ax + Bu$ and note that as a consequence, the Fourier transforms satisfy $\hat{x} \in L_2^n (i \mathbb{R})$ and $\hat{u} \in L_2^m (i \mathbb{R})$. Thus, Parseval's theorem can be applied to the integral-quadratic form in condition (iii) to obtain a corresponding expression with $\hat{x}$ and $\hat{u}$ (offset by scaled identity if $M \succeq 0$ fails to hold). Now, since $\dot{x} = Ax + Bu$ implies $\hat{x}(i \omega ) = (i\omega I - A)^{-1} B \hat{u}(i \omega )$ a.e. (and in particular not at those finite number of $\omega$ which may correspond to imaginary axis eigenvalues of $A$), by condition (iv) the integrand will be nonpositive a.e. and hence also the integral. The case $x = 0$ follows immediately by noting that the lower-right $m \times m$ block of $M$ is negative semidefinite (let $\omega \rightarrow \infty$ in condition (iv)).

\underline{$(iii) \Rightarrow (i)$:} Take any $Q(t) \succeq 0$ that satisfies (\ref{eq:lem_matrix_dynamics}) with entries in $L_1 (-\infty , \infty )$ and invoke Theorem \ref{thrm:rank_one} to obtain a rank one decomposition of $Q$ such that the vector corresponding to each term satisfies either $\dot{x}_i = Ax_i + Bu_i$ or $x_i = 0$. Note now that since all entries in $Q$ are in $L_1 (-\infty , \infty )$ and the diagonal entries are sums of squares of entries of $x_i$ and $u_i$, the integral of each such term must be convergent and it follows that $x_i \in L_2^n (-\infty , \infty )$ and $u_i \in L_2^m (-\infty , \infty )$. Consequently, condition (iii) gives
\begin{equation*}
\begin{aligned}
\int_{-\infty}^{\infty} \mathrm{tr} ( M Q(t)) \; \mathrm{dt} &= \sum_{i = 1}^{n + m} \int_{-\infty}^{\infty} \mathrm{tr} \Bigg( M \begin{pmatrix}
        x_i (t) \\ u_i (t)
    \end{pmatrix}
        (x_i (t)^T u_i (t)^T ) \Bigg) \; \mathrm{dt} \\ &= \sum_{i = 1}^{n + m} \int_{-\infty}^{\infty} \begin{pmatrix}
        x_i (t) \\ u_i (t)
    \end{pmatrix}^T M \begin{pmatrix}
        x_i (t) \\ u_i (t)
    \end{pmatrix} \; \mathrm{dt} \leq 0
\end{aligned}
\end{equation*}
for all such $Q$. With $Z = \mathcal{S}^{m + n}$, $X = \mathcal{S}^n$, $K = \mathcal{S}_{+}^{n + m}$, $(E, L)$ as in ($\ref{eq:lem_matrix_dynamics}$) and  $m^* (Q) = \mathrm{tr}(-M Q)$, this means exactly that condition (ii) in Theorem \ref{thrm:kyp_linear_cones} is satisfied. To be clear, note that we equip $X$ with the norm induced by the standard trace inner product and $Z$ with the $L_1$-norm, i.e., $\lVert \cdot \rVert_Z$ sums the absolute valued matrix entries together. Note also that any functional $f^* (Q) = \mathrm{tr}(C Q)$ for some $C \in S^{n + m}$ is clearly linear and bounded so that $m^* \in Z^*$. Invoke now Corollary \ref{cor:k_reach} to obtain K-controllability and note that $E(K) = \mathcal{S}_+^n$ so that $E(K)$ has nonempty interior. Theorem \ref{thrm:kyp_linear_cones} may thus be applied to obtain a $p^* \in X^*$, therefore on the form $p^* (Q) = \mathrm{tr}(P Q)$ for some $P \in \mathcal{S}^n$, such that the conic inequality (\ref{eq:lmi}) holds. But with $U = (A \; B)$ and $V = (I \; 0)$, since by the linearity and permutation properties of the trace operator we have
\begin{equation*}
    \begin{aligned}
        &L^*(p^*)(Q) = p^* (L(Q)) = 
            \mathrm{tr}\big( P (UQV^T + VQU^T ) \big) = \mathrm{tr}\big( PUQV^T \big) + \\ & \mathrm{tr}\big( PVQU^T \big) = \mathrm{tr}\big(V^T PUQ \big) + \mathrm{tr}\big(U^T PV Q \big) = \mathrm{tr}\big((U^T PV + V^T PU) Q \big),
    \end{aligned}
\end{equation*}
the conic inequality (\ref{eq:lmi}) means that $\mathrm{tr}((U^T PV + V^T PU + M) Q) \leq 0$ for all $Q \succeq 0$. But this is equivalent to condition (i), as $\mathrm{tr}(CQ) \geq 0$ for all $Q \succeq 0$ if and only if $C \succeq 0$.
\end{proof}


\section{Proofs} \label{sec:Proofs}
In this section, we provide proofs to the rest of the results in the previous sections.

\begin{proof} \textbf{Theorem \ref{thrm:sep_non-strict}} and \textbf{Theorem \ref{thrm:sep_strict}}

    \underline{$(i) \Rightarrow (ii)$:} Take any $z_0 \in K$ such that $L(z_0 ) = 0$ and note that $L^* (p^* ) (z_0 ) = p^* (L(z_0 )) = 0$ so that $m^* (z_0 ) \geq 0$ and $m^* (z_0 ) > 0$ follow in the non-strict and strict case, respectively.

    \underline{$(ii) \Rightarrow (i)$:} In order to find a desired $p^*$, we show the existence of a separating hyperplane. We do this first in the case of non-strict inequality, and then in the strict case. 
    
    \textit{Non-strict inequality:} Define the two convex sets $$Q = \{ (L(z) , -m^* (z) ) \mid z \in K \} \subseteq X \times \mathbb{R}$$ and $$R = \{(0, v) \mid v > 0 \} \subseteq X \times \mathbb{R},$$ both clearly convex and nonempty, and suppose that $Q$ and $R$ are not disjoint. Then there exists $\hat{z} \in K$ such that $0 = L(\hat{z} )$ and $m^* (\hat{z}) < 0$, a contradiction by condition (ii). As such, there exists a hyperplane separating the two sets, i.e., there exist a $c \in \mathbb{R}$ and a nonzero $\hat{p}^* \in (X \times \mathbb{R} )^*$, with $\hat{p}^* (x, r) = p^* (x) + q(r)$ for some nonzero pair $(p^*, q ) \in X^* \times \mathbb{R}$ by identification, such that $qv \geq c$ for all $v > 0$ and $p^* ( L(z)) -qm^* (z) \leq c$ for all $z \in K$. Now, if $c > 0$, then a sufficiently small $v > 0$ can be chosen so as to violate $qv \geq c$, and if $c < 0$, then $z = 0$ can be chosen in the second inequality to give $0 \leq c < 0$, and so $c = 0$ must hold. Further, if $q < 0$ then any $v > 0$ will violate $qv \geq c = 0$, and if $q = 0$, then $p^* \ne 0$ and $p^* (L (z)) \leq 0$ must hold for all $z \in K$, a contradiction since $L(K) = X$ by assumption. It follows that $q > 0$, and so division by $q$ gives, after the relabeling $\frac{1}{q}p^* \rightarrow p^*$ and usage of the definition of adjoints, $L^* (p^*) (z) - m^* (z) = p^* (L(z)) - m^*(z) \leq 0$ for all $z \in K$, which is equivalent to condition (i).

    \textit{Strict inequality:} Define the two sets $$Q = \{ (L(z) , -m^* (z) ) \mid z \in K \cap B \} \subseteq X \times \mathbb{R},$$ where $B = \{ z \in Z \mid \lVert z \rVert_Z = 1\}$, and $$R = \{(0, v) \mid v \geq 0 \} \subseteq X \times \mathbb{R},$$ both clearly nonempty, and suppose that $\mathrm{conv}(Q)$ and $R$ are not disjoint. Then there exist $k$ elements $w_i \in Q$, $v_c \geq 0$ and $\alpha_i \geq 0$ with $\sum_{i = 1}^{k} \alpha_i = 1$ such that 
    \begin{equation} \label{eq:conv_eq}
    (0, v_c) = \sum_{i = 1}^{k} \alpha_i w_i = (L(z_c) , -m^* (z_c ) ),
    \end{equation}
    where $z_c = \sum_{i = 1}^{k} \alpha_i z_i$ for some $z_i \in K \cap B$. Now, being the convex combination of nonzero elements in a convex cone, preservation of nonnegative linear combinations along with pointedness implies that $z_c$ belongs to $K$ and is nonzero. Condition (ii) thus gives $m^* (z_c ) > 0$, a contradiction due to (\ref{eq:conv_eq}), and it follows that $\mathrm{conv}(Q)$ and $R$ are disjoint. 
    
    In the next step, we note first that $\mathrm{conv}(Q)$ is compact. This follows since $K \cap B$ is closed and bounded and hence compact in finite dimensions, implying that $Q$ is compact as the image of $K \cap B$ under a continuous transformation, noting also that the convex hull preserves compactness in finite dimensions. Now, since in addition $R$ is closed, there must exist a strictly separating hyperplane between the two convex sets and hence between $R$ and $Q$, i.e., there exist $c \in \mathbb{R}$ and nonzero $\hat{p}^* \in (X \times \mathbb{R} )^*$, with $\hat{p}^* (x, r) = p^* (x) + q(r)$ for some nonzero pair $(p^*, q ) \in X^* \times \mathbb{R}$ by identification, such that $qv > c$ for all $v \geq 0$ and $p^* ( L(z)) -qm^* (z) < c$ for all $z \in K \cap B$. If $c > 0$, then $v = 0$ causes a contradiction so that $p^* ( L(z)) -qm^* (z) < c \leq 0$. Similarly, if $q < 0$, then a sufficiently large $v > 0$ will contradict $qv > c$. Finally, if $q = 0$, then $p^* \ne 0$ and $p^* (L(z)) < 0$ for all $z \in K \cap B$. Since the latter is a compact set and $p^*$ and $L$ are continuous, $p^* (L(z))$ achieves its maximum in the image, which must therefore be negative. Thus, properly scaled by a constant $\beta > 0$, the maximum of $-m^* (z)$ over $K \cap B$ can be added to $p^* (L(z)) < 0$ without changing the negativity so that $p^* (L(z)) - \beta m^* (z) < 0$ for all $z \in K \cap B$. It follows after dividing by either $q$ or $\beta$ depending on if $q > 0$ or $q = 0$ that, after relabeling, $p^* ( L(z)) -m^* (z) < 0$ for all $z \in K \cap B$, which must in fact hold for all nonzero $z \in K$ since $\alpha z \in K \cap B$ for a suitable $\alpha > 0$. Altogether, this means exactly that $p^* ( L(z)) -m^* (z)$ belongs to the interior of $-K^*$, see Section \ref{Section:Preliminaries}.
\end{proof}

\begin{proof} \textbf{Corollary \ref{cor:nondym}}

    \underline{$(i) \Rightarrow (ii)$:} Since $L^* (p^*) (z) = p^* (L (z))$, we have $p^* (L(z)) - m^* (z) \leq 0$ for all $z \in K$. Taking any trajectory $z \in H$ and any associated $x \in H_z$ for which $\dot{x} = L(z)$, by integrating we obtain 
    \begin{equation*}
    \begin{aligned}
    &p^* \Bigg( \int_{t_0}^{t_1} L(z(t)) \; \mathrm{dt} \Bigg) - \int_{t_0}^{t_1} m^* (z(t)) \; \mathrm{dt} = p^* (x(t_1 )) - p^* (x(t_0 )) - \int_{t_0}^{t_1} m^* (z(t)) \; \mathrm{dt} \leq 0
    \end{aligned}
    \end{equation*}
    and so $V$ can be chosen as $p^*$.

    \underline{$(ii) \Rightarrow (i)$:} Because $H_l \subseteq H$, $H$ contains $\hat{z}(t) = z_0$ for a given $z_0 \in K$ such that $L(z_0 ) = 0$. Choosing an interval with $t_0 \ne t_1$, for any $x \in H_{\hat{z}}$ we thus have $$0 = V(x(t_1 ) ) - V(x(t_0 ) ) \leq \int_{t_0}^{t_1} m^* (\hat{z}(t)) \; \mathrm{dt} = m^* (z_0 ) (t_1 - t_0 )$$ because $\dot{x} = L(z_0 ) = 0$ so that $x(t_0 ) = x(t_1 )$. Theorem \ref{thrm:sep_non-strict} now gives condition (i).

    \underline{$(i) \Leftrightarrow (iii)$:} One direction follows via condition (ii) as $x$ can be chosen such that $x(t_1 ) = x(t_0 )$ by assumption; the other follows as in the direction $(ii) \Rightarrow (i)$.
\end{proof}

\begin{proof} \textbf{Corollary \ref{cor:dissipation}}
    
     By setting $H = \cup_{[t_0 , t_1]} \{ z: [t_0 , t_1 ] \rightarrow K \mid \frac{d}{dt} E(z) = L(z)\}$ and $H_z = \{E(z)\}$ and noting that clearly $H_l \subseteq H \subseteq H_u$, this follows immediately from Corollary \ref{cor:nondym}. Further, if the dissipation inequality holds for some $V$, then the conic inequality in condition (i) holds and a new $V$ can be chosen as in the proof of Corollary \ref{cor:nondym} as $V = p^* \in X^*$. 
\end{proof}

\begin{proof} \textbf{Theorem \ref{thrm:kyp_linear_cones}}

    Suppose first $L(K) \ne X$. Then, because $L(K)$ is a convex cone, there must exist $p^* \in X^*$ such that $p(L(z)) \leq 0$ for all $z \in K$. Choose now an interior point $x_0 \in E(K)$, which exists by assumption, and another sufficiently close point $x_1 \in E(K)$ so that $p(x_1 - x_0 ) > 0$. Next, use K-controllability to find $t_1 \geq 0$ and $z(t) \in K$ satisfying $\frac{d}{dt}E(z) = L(z)$ such that $E(z(0)) = x_0$ and $E(z(t_1 )) = x_1$. Now, by the linearity and continuity of $p^*$, we have 
    \begin{equation}
    \begin{aligned}
    0 < p^*(x_1 - x_0) &= p^* \Big( E(z(t_1 )) - E(z(0 )) \Big) = \int_{0}^{t_1} p^* \Bigg( \frac{d}{dt} E(z(t)) \Bigg) \; \mathrm{dt} \\ &= \int_{0}^{t_1} p^* (L(z(t))) \; \mathrm{dt} \leq 0,
    \end{aligned}
    \end{equation} a contradiction. As such, $L(K) = X$ and Corollary 1 and Theorem 1 may be invoked below.

    \underline{$(i) \Rightarrow (ii)$:} Take any $z$ as in condition (ii). By Corollary \ref{cor:dissipation}, $(E, L)$ satisfies the dissipation inequality (\ref{eq:dis_ineq}), and so for all restrictions of $z$ to the interval $[-T , T ]$, where $T > 0$, we have 
    \begin{equation*}
    \begin{aligned}
    V(E(z (T ))) - V(E(z (-T ))) &\leq \int_{-T}^{T} m^* (z (t)) \; \mathrm{dt} \leq \int_{-T}^{T} \big \lvert m^* (z(t)) \big \rvert \; \mathrm{dt} \\ &\leq \lVert m^* \rVert_{Z^* } \int_{-T}^{T} \lVert z (t) \rVert_{Z} \; \mathrm{dt}.
    \end{aligned}
    \end{equation*}
    Condition (ii) now follows from the continuity of $V$ by letting $T \rightarrow \infty$, as $z(\infty ) = z(-\infty ) = 0$ and the improper integral in question converges since by assumption $\int_{-\infty}^{\infty} \lVert z(t) \rVert_Z \mathrm{dt} < \infty$.

    \underline{$(ii) \Rightarrow (iii)$:} Suppose on the contrary that $m^* (z_0 ) < 0$ for some $z_0 \in K$ such that $L(z_0 ) = 0$. Invoke K-controllability to construct a trajectory $z$ which is zero for $t \leq 0$ and satisfies $z(t) = z_0$ for all $t \in [1, t_2 ]$ and $z(t) = 0$ again for all $t \geq t_2 + 1$. Choose finally a large enough $t_2$ so as to make the improper integral negative and violate condition (ii) so that in fact $m^* (z_0 ) \geq 0$ and condition (iii) follows. 
    
    \underline{$(iii) \Rightarrow (i)$:} This follows directly from Theorem \ref{thrm:sep_non-strict}. 
\end{proof}

\begin{proof} \textbf{Proposition \ref{prop:bounded_real_l1}}

    Choose $Z = \mathbb{R}^{n + m}$, $X = \mathbb{R}^n$, $L(x, u) = Ax + Bu$, $E(x, u) = x$, $m^* (x, u) = \gamma \mathbf{1}_m^T u - \mathbf{1}_n^T x$ and $K = \mathbb{R}_{+}^{n + m}$. Note also that since $L(\mathbb{R}_+^{n + m})$ is a convex cone which contains part of $\mathrm{Int}(\mathbb{R}_+^{n})$ by assumption, as well as the nonpositive orthant $-\mathbb{R}_+^{n}$ ($A$ is Metzler Hurwitz so that $-A^{-1} \geq 0$), we must have $L(\mathbb{R}_+^{n + m}) = X$, so that Theorem \ref{thrm:sep_non-strict} and Corollary \ref{cor:dissipation} can be invoked below.

    \underline{$(i) \Rightarrow (ii)$:} Note first that $\int_0^{\infty} m^* (x, u) \; \mathrm{dt} \geq 0$ for all nonnegative $u \in L_{1}^m [0, \infty )$ is equivalent to the supremum part of condition (i). Suppose now that there is some $z_0 = (x_0 , u_0) \geq 0$ with $L(z_0 ) = Ax_0 + Bu_0 = 0$ such that $m^*(z_0) < 0$. We note that the system behaves like an unforced system with equilibrium point at $x_0$ when $u(t) = u_0$, as $$\dot{\tilde{x}} = \dot{x} = Ax + Bu_0 = A(x - x_0) + Ax_0 + Bu_0 = A\tilde{x}$$ where $\tilde{x} = x - x_0$. Thus, since $A$ is Hurwitz, the trajectory starting at the origin ($\tilde{x} = -x_0$) will converge to $x_0$ ($\tilde{x} = 0$) and will additionally be confined to the nonnegative orthant by positivity/monotonicity. The integral over $m^*(x, u)$ can thus be made arbitrarily negative by letting $u(t) = u_0$ sufficiently long due to $m^*(x_0, u_0) < 0$, after which it can be completed into a nonnegative $L_{1} [0, \infty )$-trajectory by setting $u(t) = 0$, a contradiction by condition (i). Hence, $m^*(z_0) \geq 0$ and so by Theorem \ref{thrm:sep_non-strict}, there is $\hat{p}^* \in X^*$, i.e., $\hat{p}^* (x) = p^T x$, such that for all $z = (x, u) \in \mathbb{R}_+^{n + m}$, 
    \begin{equation} \label{eq:functional_to_vector}
    \begin{aligned}
    0 &\geq \hat{p}(L(x, u)) - m^*(x, u) = p^T (Ax + Bu) - (-\mathbf{1}_n^T x + \gamma \mathbf{1}_m^T u) \\ &= \big( p^T A + \mathbf{1}_n^T \; \; \; p^T B - \gamma \mathbf{1}_m^T \big) z
    \end{aligned}
    \end{equation} and the two inequalities in condition (ii) follow since $\mathbb{R}_+^{n + m}$ is self-dual. Finally, $p > 0$ follows from the upper equation, as $p^T \geq -\mathbf{1}_n^T A^{-1} > 0$, since being invertible, $A^{-1}$ can have no zero columns.

    \underline{$(ii) \Rightarrow (i)$:} Note first that $A^T p \leq -\mathbf{1}_n < 0$ for $A$ Metzler and $p > 0$ implies that $A$ is Hurwitz, e.g., \cite{rantzer2018tutorial}. Take now any nonnegative $u \in L_{1}^m [0, \infty )$ with a corresponding $x \in L_{1}^n [0, \infty )$ (since $A$ is Hurwitz) such that $\dot{x} = Ax + Bu$, and note that $x(t) \geq 0$ since the system is positive. In light of (\ref{eq:functional_to_vector}), invoke Corollary \ref{cor:dissipation} to find a continuous $V$ such that the restriction of $(x, u)$ to an interval $[0, T]$ with $T > 0$ satisfies the dissipation inequality (\ref{eq:dis_ineq}). As a result, by letting $T \rightarrow \infty$ and noting that $x(0) = x(\infty ) = 0$, we have $\int_0^{\infty} m^* (x, u) \; \mathrm{dt} \geq 0$ and therefore condition (i). 
\end{proof}

\begin{proof} \textbf{Corollary \ref{cor:k_reach}}
    
    Define $E$ and $L$ through (\ref{eq:lem_matrix_dynamics}) so that $\frac{d}{dt} E(z) = L(z)$. First, for any pair $X_0, X_1 \in E(\mathcal{S}_{+}^{n + m}) = \mathcal{S}_{+}^n$, we can perform a spectral decomposition on $X_0$ and $X_1$ and exploit controllability to find control inputs connecting the vectors in each rank one term to one another over finite time, say $t_1 \geq 0$. Stacking the resulting state trajectories and their corresponding control inputs, forming rank one matrices and summing up gives a desired $z$ by Theorem \ref{thrm:rank_one} and thus controllability on $\mathcal{S}_{+}^{n+m}$.

    For the converse, take any nonzero $\hat{x} \in \mathbb{R}^n$ and invoke K-controllability to obtain a time $t_1 \geq 0$ and a trajectory $z$ such that $E(z(0 )) = 0$ and $E(z(t_1 )) = \hat{x}\hat{x}^T \succeq 0$. By Theorem \ref{thrm:rank_one}, $z$ can be expressed as a sum of rank one terms with corresponding vectors satisfying either $x_i = 0$ or $\dot{x}_i = Ax_i + Bu_i$. Since $E(z(t_1 )) = \hat{x}\hat{x}^T \ne 0$, at least one of the latter kind must exist with $x_i \ne 0$ at $t = t_1$, and if several exist they must be proportional to some vector which by extension must also be proportional to $\hat{x}$. Any of the corresponding $u_i$ may then be taken with appropriate scaling, thus proving reachability and therefore controllability. 
\end{proof}

\begin{proof} \textbf{Theorem \ref{thrm:rank_one}}

\underline{$(i) \Rightarrow (ii)$}: Given the partition $$Q = \begin{pmatrix}
        Q_{nn} & Q_{nm} \\ Q_{nm}^T & Q_{mm}
    \end{pmatrix},$$ we show first how to obtain the desired $x_i$ and $u_i$ on any open interval $I$ on which $Q_{nn}(t)$ has constant rank (Step 1). Next, we invoke a boundedness property to argue that the relevant functions are well-behaved and possess one-sided limits at the boundary of $I$ (Step 2), and in the final step we apply induction to connect the $x_i$ in a continuous manner across the whole of $\mathbb{R}$ (Step 3).
    
    \textit{Step 1:} To this end, begin by noting that according to Lemma \ref{lem:image_partition}, $\mathrm{Im}(Q_{nm}(t)) \subseteq \mathrm{Im}(Q_{nn}(t))$ for all $t \in \mathbb{R}$ so that each column of $Q_{nm}(t)$ must belong to $\mathrm{Im}(Q_{nn} (t))$. But this means that $Q_{nm} (t) = Q_{nn}(t) R(t)$, where for example $R(t) = Q_{nn}^\dagger (t) Q_{nm}(t)$, with $Q_{nn}^\dagger$ denoting the pseudoinverse. As a consequence,
    \begin{equation} \label{eq:Q_decomp}
    Q = \begin{pmatrix}
        I \\ R^T
    \end{pmatrix} Q_{nn} \begin{pmatrix}
        I & R
    \end{pmatrix} + \begin{pmatrix}
        0 && 0 \\
        0 && Q_{mm} - R^T Q_{nn} R
    \end{pmatrix}.
    \end{equation}
    and so condition (i) implies that on those open intervals $I \subseteq \mathbb{R}$ on which $Q_{nn}$ has constant rank and $R(t)$ is therefore continuous, $Q_{nn}$ satisfies
    \begin{equation} \label{eq:Qnn_dynamics}
        \dot{Q}_{nn} = (A + BR(t)^T )Q_{nn} + Q_{nn}(A + BR(t)^T)^T .
    \end{equation}
    On the other hand, if $X: I \rightarrow \mathbb{R}^{n \times n}$ is the unique solution to $\dot{X} = (A + BR^T)X = AX + BR^T X$ with initial condition $X(t_0) = \sqrt{Q_{nn} (t_0 )}$, it follows that $XX^T$ is a solution to (\ref{eq:Qnn_dynamics}), and so by uniqueness $Q_{nn} = XX^T$. As such, on these intervals $x_i$ and $u_i$ can be chosen as the $i$:th column in $X$ and $R^T X$, respectively. The remaining $m$ terms in condition (ii) may then be selected following a spectral decomposition of the second term in (\ref{eq:Q_decomp}), which is well-behaved for symmetric and real-analytic matrix functions of time \cite{rellich1969perturbation}.

    \textit{Step 2:} Since $R(t)$ may grow unbounded as a result of a rank change in $Q_{nn}(t)$, we must verify that the relevant functions are well-behaved also on the closure of $I$. To this end, note that multiplication from the left and right by $(z, w) \in \mathbb{R}^{n + m}$ in (\ref{eq:Q_decomp}) gives the quadratic form $(z + Rw)^T Q_{nn}(z + Rw) + w^T (Q_{mm} - R^T Q_{nn} R)w$ at a given $t$, which is nonnegative for all $(z, w)$ and therefore also when $z = -Rw$. It follows that $Q_{mm} \succeq R^T Q_{nn} R = R^T XX^T R$, and so $R^T X$ is bounded. Since $Q_{nn} = XX^T$ so that $X$ is also bounded, $\dot{X} = AX + BR^T X$ is also bounded and consequently, by the generalized mean value theorem, $X$ will possess one-sided limits so that $X$ can be defined such that it be continuous on the closure of $I$.
    
    \textit{Step 3:} Although the desired $x_i$ and $u_i$ may now be constructed on the closure of $I$, $x_i$ may not necessarily connect continuously across such intervals. In order to show that this is indeed possible, partition the real line into closed intervals on the interior of which $Q_{nn}$ has constant rank, take any such interval and call it $I_0$, and number the remaining ones greater than $I_0$ by $I_i$ in increasing order. Importantly, due to real analyticity, there will be no infinite oscilliations towards an accumulation point and as such, $\cup_i I_i$ will cover the part of the real line greater than $I_0$. We now proceed inductively to construct a continuous $X: \cup_i I_i \rightarrow \mathbb{R}^{n \times n}$ such that $Q_{nn} = XX^T$ and $\dot{X} = (A + BR^T )X$ in the following manner: define $X_0 : I_0 \rightarrow \mathbb{R}^{n \times n}$ and $X_1 : I_1 \rightarrow \mathbb{R}^{n \times n}$ according to the above and note that at $I_0 \leq t_* \leq I_1$, we have $Q_{nn} = X_0 X_0^T = X_1 X_1^T$. This implies the existence of a $U U^T = I$ such that $X_1 (t_* ) U = X_0 (t_* )$, which can be seen by exploiting the polar decompositions of $X_0$ and $X_1$. We may thus obtain continuity on $I_0 \cup I_1$ by instead considering $\tilde{X}_1 = X_1 U$, which also satisfies $\dot{\tilde{X}}_1 = (A + BR^T )\tilde{X}_1$ and $Q_{nn} = X_1 U U^T X_1 ^T = \tilde{X}_1 \tilde{X}_1^T$ on $I_1$. Proceeding inductively by attaching additional intervals in precisely the same manner, a desired continuous $X$ may be constructed on the part of the real line greater than $I_0$; the same argument can also be repeated for the remainder of the real line less than $I_0$, and the proof is complete.
    
    \underline{$(ii) \Rightarrow (i)$}: Since $\dot{x}_i x_i^T + x_i \dot{x}_i^T = (Ax_i + Bu_i ) x_i^T + x_i (Ax_i + Bu_i )^T$ for all $i$, each term in condition (ii) clearly satisfies (\ref{eq:lem_matrix_dynamics}). The same thus holds for the sum $Q$, and condition (i) follows. 
\end{proof}


\section{Conclusions} \label{Section:Conclusions}
In this paper, we have considered integral linear constraints on sets of $Z$-valued trajectories constrained pointwise in time to a cone, where $Z$ is a finite-dimensional normed space. Importantly, the satisfaction thereof was found to be equivalent to the existence of a bounded linear functional $p^*$ satisfying a conic inequality. Notably, the sets of trajectories amenable to this equivalence can but must not in any way be connected to dynamics. Conversely, finding such a solution to the conic inequality establishes an integral linear constraint for a number of trajectory sets, including the ones with differential constraints which are often of interest in the context of dynamical systems. Moreover, parallels were drawn to the control literature by showing that the satisfaction of the conic inequality is equivalent to the satisfaction of the dissipation inequality with linear supply rate on a cone, corresponding to the well-known connection between LMIs and the dissipation inequality with quadratic supply rate.

The above results were subsequently leveraged in order to prove both an $L_1$-gain analog in positive systems theory to the well-known bounded real lemma, as well as a non-strict version of the KYP Lemma in linear-quadratic control. This contributes to drawing further parallels between and bringing these traditionally different areas together under a linear-conic framework. Furthermore, there is perhaps also a contribution in the above proof of the KYP Lemma in comparison to other already existing proofs. The proof essentially passes through a more basic cone analog of the KYP Lemma in which the characteristic quadratic costs over $\mathbb{R}^{n + m}$ now become linear over a cone, see Theorem \ref{thrm:kyp_linear_cones}. The KYP Lemma with its associated quadratic functionals is then obtained by applying a crucial rank one decomposition to matrix trajectories on the positive semidefinite cone, see Theorem \ref{thrm:rank_one}. Although there are arguably more straightforward and direct ways in which to proceed, this approach additionally provides structure and novel insights as opposed to rote theorem verification, see the remarks in Section \ref{Section:Results}. For example, it diminishes the role of the frequency inequality in favor of an
integral quadratic constraint formulation for the purpose of better understanding the phenomenon mathematically. In addition, algebraic proofs celebrated for their brevity tend to obscure the connection to dynamics by postponing it to the end when the frequency domain enters into the picture. This may lead one to think that the desired connection between LMIs and constraints on the behavior of a system is inaccessible, when in fact in a dynamics proof it is made early on and can be observed clearly in its simplicity on cones.

For future works, it would be interesting to unite additional results under a conic framework, as well as to pursue what may be a cone analog to linear-quadratic theory. This is already achieved for the special case of the dissipation inequality, and suggested in the case of K-controllability. As seen in previous work on the topic, cone-preservance and monotonicity are fundamental assumptions that will likely play an important role to this end.

\printbibliography

\end{document}